\documentclass[12pt]{amsart}
\pdfoutput=1
\usepackage{latexsym, amsmath, amssymb, amsthm, bm}
\usepackage{calrsfs}
\usepackage{subfig}

\usepackage[mathscr]{eucal}
\usepackage{mathtools}
\usepackage{paralist}
\usepackage[numeric]{amsrefs}
\usepackage[T1]{fontenc}
\usepackage{mathptmx}
\usepackage{microtype}
\usepackage{tikz}
\usepackage{placeins}

\usepackage[colorlinks=true,linkcolor=blue,urlcolor=blue, citecolor=blue]%
    {hyperref}
\usepackage[paperwidth=8.5in, paperheight=11in, inner=2.5cm, outer=2.5cm,%
    top=2.5cm, bottom=2.5cm]{geometry}
\newtheorem{lemma}{Lemma}[section]
\newtheorem{theorem}[lemma]{Theorem}

\theoremstyle{definition}

\newtheorem{remark}[lemma]{Remark}
\newtheorem{examplex}[lemma]{Example}

\newtheorem{algorithm}[lemma]{Algorithm}
\newtheorem{defi}[lemma]{Definition}

\renewcommand{\geq}{\geqslant}
\renewcommand{\leq}{\leqslant}

\newcommand{\NN}{\ensuremath{\mathbb{N}}}
\newcommand{\PP}{\ensuremath{\mathbb{P}}} 
 
\newcommand{\RR}{\ensuremath{\mathbb{R}}} 
\newcommand{\EE}{\ensuremath{\mathbb{E}}} 
\newcommand{\FF}{\ensuremath{\mathbb{F}}}

\usepackage[foot]{amsaddr}

\title{Optimal small-set expander graphs and their codes}

\author{Tristram Bogart$^{1}$}
\address{$^{1}$ Departamento de Matemática, Universidad de los Andes, Bogotá, Colombia. ORCID: 0000-0002-1589-8584}
\email{tc.bogart22@uniandes.edu.co}

\author{Marcelo Fiori$^{2}$}
\address{$^{2}$ Instituto de Matemática y Estadística Rafael Laguardia, Universidad de la República, Montevideo, Uruguay. ORCID: 0000-0002-3732-1778.}
\email{mfiori@fing.edu.uy}

\author{Pedro Raigorodsky$^{3}$}
\address{$^{3}$ Instituto de Estadística (IESTA), Universidad de la República, Montevideo, Uruguay}
\email{pedro.raigorodsky@fcea.edu.uy}

\author{Mauricio Velasco$^{4}$}
\address{$^{4}$ Centro de Matemática, Universidad de la República, Montevideo, Uruguay. ORCID : 0000-0003-4787-7910}
\email{mvelasco@cmat.edu.uy}


\begin{document}

\begin{abstract} A left-regular bipartite graph $G$ of degree $d$ is called a {\it $(t,\alpha)$-small-set-expander} if every subset $X$ of left vertices of size at most $t$ has at least $\alpha |X|$ neighbors. Such a graph is an optimal small-set expander if small subsets have as many neighbors as possible. We characterize optimal expanders combinatorially via girth and prove the existence of $s$-optimal expanders for every $s$. We also prove that $s$-optimality yields new "transfer" lower bounds on the number of neighbors of sets of size $h\geq s$. Finally, as an application, we discuss the use of optimal small-set expanders in building good codes for key exchange protocols in post-quantum cryptography.
\end{abstract}

\medskip

\maketitle

\medskip

\section{Introduction}

A bipartite graph $G$ with parts $L\sqcup R$ is {\it left-regular of degree $d$} if every vertex in $L$ has exactly $d$ incident edges. We write $n:=|L|$ and $m:=|R|$ and associate to any such graph a binary linear code $B(G)\subseteq \FF_2^n$ defined as
\[B(G):=\left\{(x_1,\dots,x_n)\in \FF_2^{n}: \sum_{j\in N(r)} x_j = 0\text{ for all $r\in R$}\right\}\] 
where $N(z)$ denotes the set of neighbors of a vertex $z\in G$. Understanding how the combinatorics of the graph $G$ affects the properties of the code $B(G)$ is a very natural question and one of the main sources of interesting subspaces of $\FF_2^n$ from the point of view of coding theory and its applications. In particular, the following property has played a key role in the development of the interactions between coding theory, theoretical computer science and discrete mathematics since its introduction by Margulis in the late 70s,

\begin{defi}Let $\alpha,\gamma$ be real numbers. A left-regular bipartite graph $G$ of degree $d$ is called a {\it $(\gamma,\alpha)$-expander} if every subset $X\subseteq L$ of size at most $\gamma n$ has at least $\alpha |X|$ neighbors.
\end{defi}

It is known that if $G$ is a suitable expander then $B(G)$ admits extremely efficient decoding algorithms, with decoding radius determined by the expansion parameters. Furthermore it is known that good expander graphs are ubiquitous among sufficiently large graphs.  More precisely, it is immediate from our left-regularity assumption that the inequality $\alpha\leq d$ holds for every such $G$ and it is natural to ask whether there exist expanders approaching the equality. A Theorem of Sipser and Spielman~\cite{SiSp} shows that for every $\gamma\leq 1/2$ and $\epsilon>0$, a uniformly chosen $d$-regular bipartite graph is a $(\gamma, d(1-\epsilon))$-expander with probability which increases (exponentially) toward one as the number $n$ of left vertices increases to infinity. Furthermore, there exist many deterministic constructions of expander graphs~\cites{ReVaWi, Golowich24}, a problem of interest since the earliest works~\cites{Mar73,Mar88}. To the best of our knowledge most constructions aim towards building expander graph families, that is sequences of graphs $(G_n)_{n\in \mathbb{N}}$ with vertex size going to infinity which are $(\gamma,\alpha)$ expanders for fixed parameters $\gamma,\alpha\in \RR$. 

For certain applications, however, this particular asymptotic regime is not necessarily the only useful viewpoint on expanders. We would prefer to vary the number of vertices $t$ in the allowed subsets (letting $t$ be any number, typically smaller than a linear fraction of $n$) and ask for the smallest number of neighbors any such subset can have, leading to the following finer vector-valued invariant, already featured, although not prominently, in the seminal work~\cite{SiSp},

\begin{defi} Let $G$ be a left-regular bipartite graph with parts $L\sqcup R$. For an integer $t\leq n$  we define the $t$-th expansion constant $\alpha_G(t)$ of $G$ as the number
\[\alpha_G(t):=\min\left\{\frac{|N(X)|}{|X|}: X\subseteq L\text{ , } 0<|X|\leq t\right\}\]
\end{defi}

The aim of this article is to define \emph{optimal small-set expanders}, those bipartite graphs for which the expansion constants are as large as possible for small sets. We characterize these graphs combinatorially, and give some constraints on the parameter ranges in which they are allowed/guaranteed to exist. As an application we use them to define novel codes potentially useful for code-based cryptography. Our results show that optimality is not only theoretically interesting but also a property capable of yielding some desireable post-quantum security guarantees.

More precisely, we begin our investigation by observing that the expansion constants $\alpha_G(t)$ satisfy the following simple upper bounds,

\begin{lemma}\label{Lem: UpperBound} If $G$ is a left-regular bipartite graph of degree $d\geq 2$ with $n\geq m$ then the inequality $\alpha_G(2)\leq d-1/2$ holds. Furthermore if $\alpha_G(2)= d-1/2$ then for every integer $t\leq d$ we have
\[\alpha_G(t)\leq (d-1)+\frac{1}{t}.\]
\end{lemma}

Motivated by the previous inequality we say that a bipartite $d$-regular graph $G$ is an \emph{$s$-optimal small-set expander} if the expansion constants $\alpha_G(t)$ satisfy the opposite inequality $\alpha_G(t)\geq (d-1)+\frac{1}{t}$ whenever $t\leq s$, (for a given integer $s$, which is allowed to be larger than $d$). Our first Theorem shows that optimality is a surprisingly structural property. Recall that the \emph{girth} $g(G)$ of a graph is the number of vertices of the shortest simple cycle contained in $G$. 

\begin{theorem}\label{Thm: girth} Let $G$ be a left-regular bipartite graph of degree $d$. For any integer $s$, with $2\leq s\leq n$ the following properties are equivalent:
\begin{enumerate}
\item The inequality $g(G)>2s$ holds.
\item $G$ is an $s$-optimal small-set expander, that is
$\alpha_G(t)\geq (d-1)+\frac{1}{t}$ for every integer $t\leq s$.
\end{enumerate}
\end{theorem} 

Requiring for high girth for LDPC Tanner codes has been a well known paradigm in several current constructions (for instance, see \cite{Rabeti} for constructions of expanders with girth $8$), but the focus tends to be on specific target girths. To the best of our knowledge, the Theorem above formalizes this intuition explicitly for the first time. Further, because girth can be computed in polynomial time, this is a condition that can be checked and even be used as a basis to build expander graphs.

It is natural to ask whether $s$-optimal graphs exist. The following example gives several families of $2$-optimal small-set expanders from the line-point incidences of projective spaces over finite fields. Let $q$ denote a pure prime power, let $\FF_q$ be the unique field of size $q$ and let $\PP^k = \PP^k(\FF_q)$ be the set of $\FF_q$-points in $k$-dimensional projective space over $\FF_q$. By a line in $\PP^k$ we mean a one-dimensional projective subspace of $\PP^k$ defined by linear equations with coefficients in $\FF_q$.

\begin{examplex} \label{example: PS} Fix a positive integer $k$. Let $L$ be the set of lines in $\PP^k$ and let $R$ be the set of points in $\PP^k$. Let $G$ be the graph with vertex set $L\sqcup R$ and an edge $(\ell,v)$ if and only if the point $v$ belongs to the line $\ell$. The following statements hold:
\begin{enumerate}
\item The graph $G$ has the following parameters:
\[
\begin{tabular}{|c|c|}
\hline
symbol & value\\
\hline
$d$ & $q+1$\\
$n$ & $(q^{k+1}-1)/(q-1)$\\
$m$ & $[(q^{k+1}-1)(q^k-1)]/[(q-1)^2(q+1)]$\\
\hline
  \end{tabular}
\]
\item For every $t \geq 2$ the equality $\alpha_G(t)=(q+1)-\frac{t-1}{2}$ holds and in particular, setting $t=2$, we conclude that $G$ is a $2$-optimal small-set-expander.
\end{enumerate}
\end{examplex}

The graphs from the previous Example are never $3$-optimal. However, our next Theorem gives us a construction of $s$-optimal small set expanders for every $s$. Such graphs are obtained by selecting random subgraphs of any sufficiently rich family of $2$-optimal expanders such as that of the previous example. The result is an application of the probabilistic method with alterations, which can be implemented computationally:

\begin{theorem}\label{Thm: s_optimal_existence} Let $(G_k)_{k\in\mathbb{N}}$ be a family of $2$-optimal expander graphs with $G_k=(L_k\sqcup R_k; E_k)$ having left degree $d_k$ and denote $n_k:=|L_k|$ and $m_k:=|R_k|$. Assume $\frac{m_k^2}{n_k}$ is a bounded sequence and $m_k\rightarrow \infty$.
For all $s\in \mathbb{N}$, $s\geq 3$ and $c\in \RR$, $c>1$ there is an index $k^*\in \mathbb{N}$ such that, for every $k\geq k^*$ the graph $G_k$ contains a bipartite subgraph $A_k$ with parts $L(A_k)\sqcup R(A_k)$ satisfying:
\begin{enumerate}
\item $A_k$ is left-regular of degree $d_k$
\item $|L(A_k)|\geq cm_k\geq c|R(A_k)|$ 
\item $A_k$ is an $s$-optimal small-set expander.
\end{enumerate}
\end{theorem}

The previous Theorem not only shows existence of $s$-optimal small-set expanders but also that they are relatively abundant in $2$-optimal families.

\medskip

However, the graphs constructed above have a property which is often undesirable for applications: their minimum right-degree is vanishingly small (see Lemma~\ref{lem: rmin_selection} for a precise statement). To remediate this problem we study a non-linear selection regime by uniformly choosing a subset of $\lfloor cm_k^\beta \rfloor$ left-vertices and their right-neighbors (selection phase) and then successively removing all left-vertices of any cycle of length $\leq s-1$ (alterations phase). We carry out a detailed analysis of the behavior of the minimum right-degree under this process and show in Theorem~\ref{thm: rmin_main} that this quantity is \emph{of the same order as the average right-degree} $m_k^{\beta-1}$ whenever $1<\beta<1+\frac{1}{s-1}$ proving a novel concentration result for the right-degrees of such $s$-optimal expanders with a nonlinear ratio between the left and right parts, which are of interest in applications.

Returning to the general theory, it can be shown (see Lemma~\ref{lem: exps}) that the number of left vertices $n$ of any $s$-optimal expander of left-degree $d$ must grow exponentially (at least faster than $(d-1)^s$) and therefore $s$-optimality cannot be expected to hold for too large a value of $s$ for given $n$. 

As mentioned earlier, the expansion constants $\alpha_G(s)$ control the basic properties of the code $B(G)$ and $s$-optimality guarantees that they are large when $t\leq s$. With an eye towards applications it is natural to ask whether $s$-optimality implies lower bounds on the expansion constants $\alpha_G(h)$ for larger $h$. Our next Theorem answers this question affirmatively, allowing us to "transfer" lower bounds on the expansion constants $\alpha_G(s)$ for small $s$ to lower bounds for expansion constants $\alpha_G(h)$ for larger values of $h$. This result is especially useful because will be often possible to guarantee $s$-optimality for small $s$ and then use the Theorem to achieve lower bounds on the expansions constants for $h\geq s$ at ranges $h$ in which exhaustive exploration is simply unfeasible. The proof is based on ideas from linear optimization introduced in~\cite{CCLO} which we believe are of independent interest.

\begin{theorem}[Transfer bounds on expansion constants] \label{Thm: transfer} Let $G$ be a left-regular bipartite graph of degree $d$ and let $s,h$ be positive integers with $\max(s,h)\leq n$. If $h\geq s$ then the inequality
\[\alpha_G(h)\geq d-\frac{h-1}{s-1}\left(d-\alpha_G(s)\right)\]
holds. In particular, if $G$ is an $s$-optimal small-set expander of left degree $d$ then the inequality 
\[\alpha_G(h)\geq d-\frac{h-1}{s}\] holds for every $h\geq s$.
\end{theorem}

The graphs of Example~\ref{example: PS} show that the above inequality is sharp when $s=2$ and $d=q+1$ for any prime power $q$.

Finally, in Section~\ref{Sec: PQ} we discuss an extended application: the usage of $s$-optimal expander codes for constructing post-quantum key exchange protocols. We give a detailed introduction to the topic and show that $s$-optimality serves as a guarantee against several common cryptographic attacks on such protocols. The transfer bounds above are used to guarantee the computational feasibility of the exchange procedure in the ranges of interest. These results have convinced us that optimal small-set expanders could be of considerable practical interest.

\section{Optimal small set expanders and girth.}

\begin{proof}[Proof of Lemma~\ref{Lem: UpperBound}] Since $G$ is left-regular and has left-degree $d$, double-counting the total number of edges of $G$ gives the equality $dn = \sum_{r\in R} |N(r)|$. As a result the average right-degree $\overline{e}$ satisfies $dn/m =\overline{e}$ which, because $n\geq m$, implies that there exists a right vertex $r^*$ for which $|N(r^*)|\geq d$. Since $d\geq 2$ there must exist a right vertex with at least two distinct left vertices $v_1,v_2$. A a result $|N(\{v_1,v_2\})|\leq 2d-1$ and therefore $\alpha_G(2)\leq d-1/2$ as claimed.

Assume that $\alpha_G(2) = d-1/2$ holds. Given $t\leq d$ choose $t$ left vertices $v_1,\dots, v_t$ among the at least $d$ neighbors of $r^*$. Since $G$ is $2$-optimal every two left vertices have at most one common neighbor so $r^*$ is the only intersection point of any subset of the $v_i$ of cardinality at least two and we conclude that the number of right-neighbors of the set $X:=\{v_1,\dots, v_t\}$ is equal to $td-(t-1)$. Dividing by $|X|=t$ we conclude that $\alpha_G(t)\leq d-\frac{t-1}{t}$ as claimed.\end{proof}
\begin{remark} The proof shows that the previous inequality holds for all $t\leq e$ where $e$ is the largest degree of a right vertex of $G$.
\end{remark}

\begin{proof}[Proof of Theorem~\ref{Thm: girth}]
$(1)\implies (2)$ If $g(G)>2s$ then any nonempty set $X$ of $t\leq s$ left vertices defines an induced bipartite subgraph $H_X$ with vertices $X\sqcup N(X)$ and having $dt$ edges which connect $X$ with its neighbors in $G$. Since $H_X$ is bipartite, any simple cycle on $H_X$ would have at most $t$ vertices contraddicting our assumption on the girth of $G$. We conclude that $H_X$ must have no cycles and thus be a forest. As a result the difference between the number of vertices and edges of $H_X$ equals the number of connected components of $H_X$ and in particular it is strictly positive because $X$ is nonempty. Explicitly, we have shown that the following inequality holds
\[\left(|X|+|N(X)|\right)-d|X|\geq 1\]
or equivalently that 
\[\frac{|N(X)|}{|X|}\geq d-1+\frac{1}{t}\]
as claimed proving $(2)$. Conversely, we show $(2)\implies (1)$ by verifying that if $g(G)\leq 2s$ then $G$ is not $s$-optimal. If $g(C)\leq 2s$ then there exist left vertices $X=\{v_1,\dots, v_t\}$ and right vertices $r_1,\dots, r_t$ such that $v_1,r_1,v_2,r_2,\dots, v_t,r_t$ forms a simple cycle of length $2t$ for some $t\leq s$. Since each $r_i$ is double-counted exactly once when adding the neighborhoods of the $v_i$ we conclude that 
\[|N(X)|\leq dt-t=(d-1)t\]
so $\alpha_G(t)\leq d-1$ and therefore $G$ is not an $s$-optimal expander as we wanted to show.
\end{proof}

\begin{remark} If $G$ is $s$-optimal then there may be sets $X$ of size $t\leq s$ for which the inequality $\alpha_G(t)\geq d-1+\frac{1}{t}$ is strict. The Euler characteristic computation used in the proof above proves that if the set $X$ has size $t$ and the graph $H_X$ has $k$ connected components then the equality
\[\frac{|N(X)|}{|X|} = d-1+\frac{k}{t}\]
holds. This observation allows us to compute the expansion constants of $s$-optimal graphs for $t\leq s$. If $G$ is $s$-optimal and $0<t\leq s$ then 
\[\alpha_G(t)=d-1+\frac{k}{t}\]
where $k$ is the minimum number of connected components of $H_X$ as $X$ ranges over sets of size $t$.
\end{remark}

\section{Constructing \texorpdfstring{$s$-optimal}{s-optimal} expanders from \texorpdfstring{$2$-optimal}{2-optimal} families via random selection.}

We begin by introducing some preliminaries from combinatorics. For a positive integer $t$, let $D_t$ denote the dihedral group of symmetries of the regular polygon with $t$ vertices labelled $1,\dots, t$. If $X$ is any finite set then the group $D_t$ acts on sequences of $t$ distinct elements of $X$ via 
\[\sigma\cdot(x_1,\dots, x_t):=(x_{\sigma(1)},\dots x_{\sigma(t)}).\]
The \emph{round table orderings of $t$ distinct elements of $X$}, denoted ${\rm RTO}(X,t)$ is the set of orbits of this action. If we denote the orbit of $(x_1,\dots, x_t)$ by $[x_1,\dots, x_t]$ then $[x_1,x_2,\dots,x_{t-1}, x_t] = [x_2,x_3,\dots, x_k,x_1]$ and $[x_1,x_2,\dots,x_{t-1},x_t]=[x_1,x_t,x_{t-1},\dots, x_2,x_1]$. The class $[x_1,\dots, x_t]$ consists of all placements of $x_1,\dots, x_t$ in a round table in such a way that the neighbors of $x_i$ are $\{x_{i+1},x_{i-1}\}$ for every $i$, where the operations in the subindices are done modulo $t$.

It is immediate that
\[\left|{\rm RTO}(X,t)\right| = \frac{t!\binom{|X|}{t}}{2t}=\binom{|X|}{t}\frac{(t-1)!}{2} \]
because each orbit has the same size as the group since all $x_i$ are assumed distinct.

We will use round table orderings as a mechanism for cycle counting as follows. Formally, a simple cycle of length $t$ in any graph $H$ is a sequence $(v_1,\dots, v_t)$ of distinct adjacent vertices. However the cycle itself (i.e. the edges) defined by $(v_1,\dots, v_t)$ is the same as the cycle determined by $\sigma\cdot(v_1,\dots, v_t)$ for $\sigma\in D_t$ and therefore it is natural to count cycles using round table orderings. For an integer $t$ we define the simple cycles of length $t$ in $H$ as
\[C_t(H):=\left\{[v_1\dots v_t]\in {\rm RTO}(V(H),t): \{v_i,v_{i+1}\}\in E(H)\right\}\]
and the number of simple cycles of length $t$ is the cardinality $|C_t(H)|$ of this set.

\begin{lemma}\label{lem: expectedcycles} Let $G=(L\sqcup R ; E)$ be a $2$-optimal bipartite expander. Fix $c\in \RR$ with $c>1$ and $n\geq \lfloor cm\rfloor$. Let $H$ be the random subgraph of $G$ obtained by selecting a set $X$ of $\lfloor cm\rfloor$ vertices of $L$ uniformly at random and joining them with all their right neighbors in $G$. If $t\geq 2$ and $Z_{(t)}$ denotes the number of simple cycles of length $2t$ in $H$ then the following statements hold:
\begin{enumerate}
\item The inequality
\[\EE[Z_{(t)}]\leq \frac{(t-1)!}{2}\frac{\binom{m}{t}\binom{\lfloor cm\rfloor}{t}}{\binom{n}{t}}\]
holds, and
\item If both $m$ and $n$ are large when compared to $t$ then the upper bound in $(1)$ has the same asymptotic behavior as 
\[\frac{c^t}{2t}\left(\frac{m^2}{n}\right)^t\]
\end{enumerate}
\end{lemma}
\begin{proof} If $[\ell_1,r_1,\dots, \ell_t,r_t]\in C_{2t}(G)$ with $\{\ell_1,\dots,\ell_t\}\subseteq L$ and $\{r_1,\dots, r_t\}\subseteq R$ then both sets have $t$ distinct elements and the class $[r_1,\dots, r_t]\in {\rm RTO}(R,t)$ is well defined. Since $G$ is $2$-optimal the vertex $\ell_i\in L$ is the unique vertex of $G$ simultaneously adjacent to $r_i$ and $r_{i+1}$ (with addition mod $t$) and therefore the class $[r_1,\dots, r_t]$ determines $[\ell_1,r_1,\dots, \ell_t,r_t]$ uniquely. We conclude that the inequality $\left|C_{2t}(G)\right|\leq \left|{\rm RTO}(R,t)\right|$ holds.
 
The random variable $Z_{(t)}$ is given by
\[Z_{(t)}=\sum_{[\ell_1,r_1,\dots, \ell_t,r_t]\in C_{2t}(G)}1_{\{\ell_1,\dots,\ell_t\}\in X} \]
and therefore 
\[\EE[Z_{(t)}]=|C_{2t}(G)|\PP\{\text{Any fixed $t$-subset of $L$ belongs to $X$}\}.\]
Since the selection of $X$ is done uniformly at random the probability term equals $\frac{\binom{\lfloor cm\rfloor}{t}}{\binom{n}{t}}$. By the inequality $\left|C_{2t}(G)\right|\leq \left|{\rm RTO}(R,t)\right|$ proved in the previous paragraph we conclude that
\[\EE[Z_{(t)}]\leq \left|{\rm RTO}(R,t)\right|\frac{\binom{\lfloor cm\rfloor}{t}}{\binom{n}{t}}\leq \frac{(t-1)!}{2}\frac{\binom{m}{t}\binom{\lfloor cm\rfloor}{t}}{\binom{n}{t}}\]
proving $(1)$. For the asymptotics we assume that both $n$ and $m$ are large and $t$ is fixed so we can approximate binomials via expressions of the form $\binom{m}{t}\sim \frac{m^t}{t!}$ yielding
\[\frac{(t-1)!}{2}\frac{\binom{m}{t}\binom{\lfloor cm\rfloor}{t}}{\binom{n}{t}}\sim \frac{(t-1)!}{2} \frac{\frac{m^t}{t!}\frac{(cm)^t}{t!}}{\frac{n^t}{t!}}=\frac{c^t}{2t}\left(\frac{m^2}{n}\right)^t\]
as claimed in $(2)$.
\end{proof}

We are now ready to prove the main result of this Section.

\begin{proof}[Proof of Theorem~\ref{Thm: s_optimal_existence}] Let $c'=2c$. Our assumptions on the sequences $m_k,n_k$ and the asymptotics from Lemma~\ref{lem: expectedcycles} part $(2)$ ensure that there exists an index $k_1$ such that $k\geq k_1$ implies that the following inequality holds
\[\sum_{t=3}^s\frac{(t-1)!}{2}\frac{\binom{m_k}{t}\binom{\lfloor c'm_k\rfloor}{t}}{\binom{n_k}{t}}\leq 2\sum_{t=3}^s \frac{(c')^t}{2t}\left(\frac{m_k^2}{n_k}\right)^t\]
and that the right-hand side is bounded by a constant $M$ for all such $k$. Furthermore, because $m_k\rightarrow \infty$ and $c>1$ there exists $k_2$ such that $k\geq k_2$ implies 
\[c'm_k-sM\geq cm_k.\] 

We claim that if $k^*=\max(k_1,k_2)$ then for every $k\geq k^*$ there exists an $s$-optimal expander $A_k\subseteq G_k$ satisfying the conclusions stated in the Theorem, and we will prove it by using the probabilistic method with alterations. 

If $H_k$ is the random subgraph of $G_k$ obtained by selecting a set $X_k$ of $\lfloor c'm_k\rfloor$ vertices of $L_k$ uniformly at random and joining them with all their neighbors in $G_k$ and $Y^{(k)}$ is the random variable counting the number of simple cycles of length $3\leq t\leq s$ in $H_k$ then, using the notation of Lemma~\ref{lem: expectedcycles} we have $Y^{(k)}:=\sum_{t=3}^s Z_{(t,k)}$ where $Z_{(t,k)}$ denotes the number of simple cycles of length $2t$ in $H_k$ and therefore 
\[\EE[Y^{(k)}]=\sum_{t=3}^s\frac{(t-1)!}{2}\frac{\binom{m_k}{t}\binom{\lfloor c'm_k\rfloor}{t}}{\binom{n_k}{t}}\leq M\]
so there must exist at least one realization of $H_k$ which we denote $B_k\subseteq G_k$ having at most $M$ such cycles. Since every cycle has at most $t\leq s$ left vertices, we can remove all left-vertices of $B_k$ involved in any such cycle and let $A_k$ be the subgraph of $B_k$ consisting of the remaining left vertices together with all their neighbors in $G_k$. 
Since $B_k$ has $c'm_k$ vertices the graph $A_k$ has at least $c'm_k-sM\geq cm_k$ vertices and each of them has left-degree $d_k$. The graph $A_k$ has no cycles of length $\leq 2s$, because we have removed them, and is therefore $s$-optimal as claimed.
\end{proof}

The previous proof yields us an algorithm to be used to generate such graphs, as is standard when applying the probabilistic method with alterations. More precisely,

\begin{algorithm}\label{alg: generator}
    INPUT: $G = (L \sqcup R, E)$ left $d$-regular, $2$-optimal bipartite expander; $s \in \NN, s \geq 3$; $N \in \NN$ with $N \leq |L|$. 
    
    OUTPUT: $G'=(L',R)$ left $d$-regular, $s$-optimal bipartite expander.
    \begin{enumerate}
        \item Generate $L'$ a random subset of $L$ of size $N$.
        \item $G':= (L' \sqcup R, E_{L' \rightarrow R})$ the induced subgraph.
        \item \textbf{While} $g(G') \leq 2s$ \textbf{do:}
        \item \quad $C:=$ shortest cycle of $G'$. Denote $\{\ell_1,\hdots,\ell_j\}$ all left nodes of $C$.
        \item \quad $L' := L' \backslash \{\ell_1,\hdots,\ell_j\}$.
        \item \quad $G' := (L' \sqcup R, E_{L' \rightarrow R})$
        \item Return $G'$
    \end{enumerate}
\end{algorithm}

Naturally, without control of the number of short cycles, our algorithm can completely destroy $G$. However, we have shown that with high probability - in the context of Theorem \ref{Thm: s_optimal_existence} - this algorithm will generate a desired graph provided $k$ is large enough. Also note that because we are always computing the shortest cycle, we always eliminate chordless cycles, so each iteration will only remove at most $sd$ edges, and the right degree of each node in the cycle will decrease by exactly $2$. This gives our algorithm stability properties that will be useful in Section \ref{Sec: RD}.

Our next result is about the limits of $s$-optimality. We show that any $s$-optimal small-set expander of degree $d$ must have a number of vertices which grows exponentially in $s$, more precisely

\begin{lemma}\label{lem: exps} If $G$ is an $s$-optimal expander of left-degree $d$ with $n\geq m$ then $n\geq \sum_{j=0}^{s+1}(d-1)^j$.
\end{lemma}
\begin{proof} Double-counting edges we have $dn=m\overline{e}$ where $\overline{e}$ is the average right-degree and therefore the average degree $\overline{d}$ of our graph $G$ is at least $d$. By Theorem~\ref{Thm: girth} $G$ is $s$-optimal implies $g(C)\geq 2(s+1)$
Alon, Horray and Linial prove in~\cite{HooryLinialWigderson06}[Section 2] that if $G=(L\sqcup R;E)$ is a bipartite graph with average degree $\overline{d}$ and girth $g(G)=2t$ then the inequality 
\[|L|+|R|\geq 2\sum_{j=0}^{s+1}(\overline{d}-1)^j\]
from which the claimed inequality 
\[2n \geq 2\sum_{j=0}^{s+1}(d-1)^j\]
follows.
\end{proof}

\section{On the behavior of right-degrees}\label{Sec: RD}
A right-vertex $r\in R(G)$ with $e$ neighbors determines a \emph{dual word}  of weight $e$, that is an equation satisfied by all words of $B(G)$, namely $\sum _{j\in N(r)} x_j=0$. The existence of dual words of small weight is a natural source of cryptographic attacks on $B(G)$ and is therefore natural to ask for lower bounds for the minimum right-degree of any expander. In this Section we focus on deriving such bounds for expanders built via random selection and alteration from $2$-optimal families.

Our first result generalizes Theorem~\ref{Thm: s_optimal_existence}. It constructs $s$-optimal expanders by a random selection of left vertices which scales like $cm_k^{\beta}$ for any value of a new parameter $\beta$ which satisfies the inequality $1\leq \beta <1 +\frac{1}{s-1}$ and then modifies the resulting subgraph to remove short cycles.  Although the result looks similar to Theorem~\ref{Thm: s_optimal_existence}, we will see that the behavior of the minimum right-degree of the resulting $s$-optimal expanders $A_k$ is radically different when $\beta=1$ (linear) and $\beta>1$ (non-linear) case. More precisely it will turn out that  as $k\rightarrow \infty$, the minimum right-degree goes to zero when $\beta=1$ while it is bounded below by $\frac{dcm_k^{\beta-1}}{2}$ with probability going to one.

\begin{theorem}\label{Thm: existance_nonlinear} Let $(G_k)_{k\in\mathbb{N}}$ be a family of $2$-optimal expander graphs with
$G_k = (L_k \sqcup R_k; E_k)$ having left degree $d_k$. Let $n_k := |L_k|$ and $m_k := |R_k|$. Assume that the sequence $\frac{m_k^2}{n_k}$ is bounded and that $m_k \to \infty$. For all $t \in \mathbb{N}$, $t \ge 3$, all $c \in \mathbb{R}$ with $c > 1$, and all
$\beta \in \RR$ with $1 \leq \beta < 1 + \frac{1}{s-1}$, there exists an index
$k^\ast \in \mathbb{N}$ such that, for every $k \ge k^\ast$, the graph $G_k$
contains a bipartite subgraph $A_k$ with parts
$L(A_k) \sqcup R(A_k)$ satisfying:
\begin{enumerate}
  \item $A_k$ is left-regular of degree $d_k$;
  \item $|L(A_k)| \ge c m_k^{\beta} \ge c |R(A_k)|^{\beta}$;
  \item $A_k$ is an $s$-optimal small-set expander.
\end{enumerate}
\end{theorem}
\begin{proof} If $G$ is any $2$-optimal graph and $H$ be a random subgraph of $G$ obtained by selecting a set of vertices
$X \subseteq L$ with $|X| = \lfloor c m^{\beta} \rfloor$ and joining them with all their right neighbors in $G$.
If $t \geq 2$ and $Z_{(t)}$ denotes the number of cycles of length $2t$ in $H$, then arguing as in Lemma~\ref{lem: expectedcycles} we have
\[
\EE[Z_{(t)}] \leq \frac{(t-1)!}{2}
\frac{\binom{m}{t}\binom{\lfloor c m^{\beta} \rfloor}{t}}{\binom{n}{t}}= O\bigl(m^{(\beta-1)t}\bigr)
\]
Where the last equality holds if $n \sim m^2$.
Now assume the $(G_k)_{k\in \NN}$ are a $2$-optimal family as above. Pick $c' = 2c$. Following the same steps of the proof of Theorem~\ref{Thm: s_optimal_existence}, it suffices to show that there exists
$k^{*} \in \NN$ such that if $k \geq k^{*}$, then
\[
c' m_k^{\beta} - s \EE[Y^{(k)}] > c m_k^{\beta}.
\]
To prove this recall that
\[
Y^{(k)} = \sum_{t=3}^{s} Z_{(t,k)}
\]
is the total number of cycles of length at most $2s$ in the random subgraph $H_k$ of $G_k$ obtained by selecting
a set $X_k = \lfloor c' m_k^{\beta} \rfloor$ vertices from $L_k$ and joining them with their neighbors in $R_k$. 

The condition $\beta < 1 + \dfrac{1}{s-1}$ is equivalent to $(s-1)(\beta - 1) < 1$.
Therefore,
\[
\EE[Y^{(k)}]
= \sum_{t = 3}^{s} O\!\left(m_k^{(t-1)(\beta-1)}\right)
= O\!\left(m_k^{(s-1)(\beta-1)}\right)
= O\!\left(m_k^{\beta}\right).
\]
Hence, there exists $k^{*}$ large enough such that for all $k \geq k^{*}$,
\[
(c' - c)m_k^{\beta} \geq s \EE[Y^{(k)}],
\]
as claimed.\end{proof}

\begin{remark} 
We observe that this exponent ($\beta = 1+\frac{1}{s-1}$) is in fact optimal. Following the arguments for Moore bounds in irregular graphs (see ~\cite{AHL}*{pags. 54-57},  or laid out explicitly in  ~\cite{Hoory}), it follows that for any family $G'_k = L'_k \sqcup R_k'$ of $s$-optimal, left $d$-regular bipartite graphs we have that $|L'_k| \lesssim |R'_k|^{1+\frac{1}{s-1}}$. Indeed, if $d,d_R$ are the average left and right degrees respectively, the leading-order growth in the Moore bound argument yields

\[|R'_k| \gtrsim d^{s}d_R^{s-1} = d^{s}\cdot \dfrac{d^{s-1}|L_k'|^{s-1}}{|R'_k|^{s-1}},\]

and the argument concludes by isolating $|L_k'|$. This is relevant because in general we want the left-to-right ratio to be large in the context of coding theory, since for instance it will maximize the average right degree (see section \ref{Sec: PQ}). It is also worth noting that the same exponent can be reached by following the cycle removal argument on random left $d$-regular graphs, although naturally we have less control over the structural properties of these expanders.
\end{remark}

To capture the behavior of the minimum right-degree we separate our analysis into two parts. In the first Lemma we focus on the behavior of this quantity under random selections and on the key dichotomy between the linear $\beta=1$ and nonlinear $\beta>1$ cases.

For each $r\in R_k$ let $d_0(r)$ denote the right-degree of the vertex $r$ in the random subgraph obtained by uniformly choosing a subset of left vertices and all their neighbors and let $d_0^*:=\min_{r\in R} d_0(r)$. 

First we study the minimum right degree during the left node selection (ignoring, for now, the alterations phase) from a $2$-optimal family which is regular on both sides,

\begin{lemma}\label{lem: rmin_selection}
Let $(G_k)_{k \in \NN}$, $G_k = (L_k \sqcup R_k,E_k)$ be a family of $2$-optimal, $d$-left regular and $\delta$-right regular graphs such that $|L_k| = n_k \sim m_k^2 = |R_k|^2$. For each $k$, let $X_k \subseteq L_k$ be a uniformly chosen random subset of size $\left\lfloor c\cdot m_k^{\beta}\right \rfloor<n_k$ and let $H_k$ be the induced bipartite subgraph with vertices $X_k\sqcup N(X_k)$. The following statements hold,
\begin{enumerate}
    \item If $\beta>1$, then for all $a < c \cdot d$ we have 
    \[\lim_{k\rightarrow \infty}\Pr\left(d_0^* \geq a \cdot  m_k^{\beta-1}\right) = 1\]

    \item If $\beta = 1$ then the random variable
    $d_0^*$ converges to zero in probability as $k\rightarrow\infty$. 
\end{enumerate}
\end{lemma}
\begin{proof} Let $r\in R_k$ be any right vertex. Since the set $X_k$ is selected uniformly at random, the random variable $d_0(r)$ will have value $j$ when exactly $j$ of the $\delta$ neighbors of $r$ belong to the random subset $X_k$. Because the set is selected unifornly at random this random variable has hypergeometric distribution with population size $N=n_k$, having $M=\lfloor cm_k^\beta\rfloor$ successes and $D=\delta$ trials. The mean of this distribution is the quantity $D(M/N) = \delta \frac{cm_k^{\beta}}{n_k}=c\cdot d (m_k^{\beta-1})=:\mu_k$
To prove $(1)$ recall that the Chernoff bound states that for any positive $\eta>0$ the hypergeometric distribution satisfies  
\[\PP\left\{d_0(r)\leq (1-\eta)\mu_k\right\}\leq \exp\left(-\frac{\eta^2}{2}\mu_k\right)\]
From this and the union bound, it follows that if $\beta>1$ then for any $\eta>0$ we have
\[
\mathbb{P}\!\left\{d_0^*\le (1-\eta)\mu_k\right\}
\le \sum_{r\in R_k}\mathbb{P}\!\left\{d_0(r)\le (1-\eta)\mu_k\right\}
\le m_k \exp\!\left(-\frac{\eta^2}{2}\mu_k\right)
\xrightarrow{k\to\infty} 0.
\]
proving the claim by an appropriate selection of $\eta$.

To show $(2)$ observe that if $dc<1$, then there must exist some isolated $r \in R_k$ following the selection, because $dc$ is precisely the average right degree, so the result holds trivially (in fact, the minimum degree is $0$ always). So we assume $dc \geq 1$.

Let $I_r := \mathbf{1}_{\{d_0(r)=0\}}$ be the random variable which indicates that $r$ is an isolated vertex and let $\chi_k = \sum_{r \in R_k} I_r$. If $\beta=1$ then $\mu_k$ is asymptotically constant to $c\cdot d$ and the random variables $d_0(r)$ have the same distribution for every $r\in R_k$ and in particular have a positive probability of being identically zero. This probability can be easily estimated from the poisson approximation to the hypergeometric distribution to be $e^{-c\dot d}$ and it follows that $\EE[\chi_k]\sim m_ke^{-cd}$ so the expected value of $\chi_k$ diverges as $k\rightarrow \infty$. However, this is not enough to conclude that $d_0^*=0$ in probability because of correlations (for a counterexample imagine the case in which all the random variables $d_0(r)$ not only have the same distribution but are copies of the same random variable). 
We claim that we can bound the covariance terms well enough so as to ensure that the ratio $\frac{\mathrm{Var}(\chi_k)}{\EE[\chi_k]^2}\rightarrow 0$ as $k\rightarrow \infty$. This is useful because by Chebyshev's inequality,
\[
\Pr\!\left(
\left|\chi_k - \EE[\chi_k]\right|
\ge \frac12 \EE[\chi_k]
\right)
\le
\frac{4\,\mathrm{Var}(\chi_k)}
     {\EE[\chi_k]^2}
\to 0.
\]
Hence with probability tending to $1$,
$\chi_k \ge \frac12 \EE[\chi_k]
\sim \frac12 e^{-cd} m_k$ and in particular,
$P(\chi_k \ge 1) \to 1$ as claimed. 

We now bound $\mathrm{Var}(\chi_k)$. We have
\[
\mathrm{Var}(\chi_k)
=
\sum_{r\in R_k} \mathrm{Var}(I_r)
+
\sum_{\substack{r,r'\in R_k\\ r\neq r'}} 
\mathrm{Cov}(I_r,I_{r'}).
\]

\medskip

Since $I_r$ is Bernoulli with parameter $p_0^{(k)}$ we have
\[
\sum_{r\in R_k} \mathrm{Var}(I_r)
\le
m_k p_0^{(k)}
=
\EE[\chi_k].
\]
For the covariance, fix two distinct right vertices $r \neq r'$.
The event $I_r=1$ means that none of the left neighbors of $r$
are selected in $X_k$.
Let $N(r)$ denote the set of left neighbors of $r$.
Because $G_k$ is $2$-optimal, we have $|N(r) \cap N(r')| \le 1$, and hence
\[
|N(r)\cup N(r')|
=
2\delta - y,
\qquad y \in \{0,1\}.
\]

\medskip

We will show that $\text{Cov}(I_r,I_{r'})=O(1/m_k)$. We will assume that $y=1$, and the other case works exactly the same.

\medskip

It is known that the total variation distance between a hypergeometric $\mathcal{H}(M,D,N)$ and a binomial distribution $Bi(M,D/N)$ is $O(M/N)$ (see \cite{Ehm}*{Theorem 1, Lemma 2}) provided $M \cdot (D/N) (1-D/N) \geq 1$. Here $M = cm_k$, $D = dn_k/m_k$ and $N = n_k$, so the condition holds for large enough $k$ since $dc \geq 1$. It follows that

\[\Pr(I_r = 1) = \left(1-\dfrac{\delta}{n_k}\right)^{cm_k} +O(cm_k/n_k) = \left(1-\dfrac{\delta}{n_k}\right)^{cm_k} +O(1/m_k)\]

Similarly,

\[\Pr(I_r = 1,I_{r'}=1) = \left(1-\dfrac{2\delta-1}{n_k}\right)^{cm_k} +O(1/m_k)\]

\medskip

Now we use that $(1+a/x)^x = e^a+O(1/x)$. Because  $c\cdot m_k\delta/n_k = cd$, and also $n_k/\delta = \frac{1}{d} \cdot m_k$ we have,

\[\Pr(I_r = 1) = \left(1-\dfrac{1}{n_k/\delta}\right)^{(n_k/\delta)\cdot  cm_k \cdot \delta /n_k} +O(1/m_k)= (e^{-1}+O(d/m_k))^{cm_k\delta/n_k}+O(1/m_k) = e^{-dc}+O(1/m_k) \]
where the last equality holds since $dc > 1$. Similarly,
\[\Pr(I_r = 1,I_{r'}=1) = e^{-2cd}+O(1/m_k)\]
We have thus shown that
\[\text{Cov}(I_r,I_{r'}) = \Pr(I_r = 1,I_{r'}=1) - \Pr(I_r = 1)^2 = e^{-2cd}+O(1/m_k)-(e^{-cd}+O(1/m_k))^2 = O(1/m_k)\]
and therefore,
\[
\sum_{r\neq r'} \mathrm{Cov}(I_r,I_{r'})
=
O\!\left(m_k^2 \cdot \frac{1}{m_k}\right)
=
O(m_k).
\]
Since $\EE[\chi_k] \sim m_k$, combining all of the above we have
\[
\frac{\mathrm{Var}(\chi_k)}
     {\EE[\chi_k]^2}
=
O\!\left(\frac{1}{m_k}\right)
\to 0.
\]
\end{proof}

Finally, for the nonlinear case $1<\beta<1+\frac{1}{s-1}$ we show that the alterations do not affect our conclusion too substantially. Recall that the removal of each cycle decreases the degree of the right vertices of the cycle by two and is therefore possible that if a right-vertex $r$ belongs to many such cycles, its degree decreases significantly after cycle removal dragging down the minimum degree. The key point however is that the removal of a cycle implies the removal of \emph{all} its left vertices and thus the alterations phase of the algorithm will only effectively remove a left-vertex-disjoint subset of all the cycles through $r$. The following Lemma shows that the probability of removing a fixed fraction of the expected degree from any right vertex vanishes asymptotically,

\begin{lemma} \label{lem: rmin_alterations} Let $(G_k)_{k\in\mathbb N}$ be a family of $d$-left regular $2$-optimal bipartite graphs such that $n_k \sim m_k^2$. Let $X_k \subseteq L_k$ be chosen uniformly at random of size 
$\lfloor c m_k^\beta\rfloor$ where $c>0$ and $\beta>1$.
If $1<\beta<1+\frac{1}{s-1}$ then for any $\eta>0$ the probability that any execution of the alterations phase of Algorithm \ref{alg: generator} removes at least $2\eta m_k^{\beta-1} $ edges from any right vertex $r$ vanishes as $k\rightarrow \infty$. 
\end{lemma}

\begin{proof} For a vertex $r\in R_k$ and an integer $j$ let $N^j(r)$ be the event that $X_k\supseteq A$ for some set $A$ which is a union of the left-vertices of a collection of $j$ left-vertex disjoint cycles with lengths in the interval $[6,2s]$. 

We study the sets $N^j(r)$ because every execution of the algorithm in which the degree of $r$ drops by at least $2j$ in the alterations phase requires the removal of a set $A$ of $j$ left-vertex disjoint cycles with $A\subseteq X_k$. This is because the algorithm removes a cycle by removing its edges and \emph{all} its left-vertices, so it will only remove other pairs edges coming out of $r$ if their corresponding cycles are left-vertex disjoint. It follows that $\PP(N^j(r))$ is an upper bound on the probability that our algorithm decreases the right-degree of $r$ by $2j$ or more. 

We will derive an upper bound on its probability via a union bound over the possible $A$'s. More precisely if $b_i$ denotes the number of cycles of size $2i$ in such a collection then the set $A$ has size $a=\sum_{i=3}^s ib_i$, $j=\sum_{i=3}^sb_i$ and 
\[\PP\{X_k\supseteq A\}= \frac{\binom{n_k -a}{cm_k^\beta-a}}{\binom{n_k}{cm_k^\beta}}\]
and furthermore the number of such collections is bounded above by
\[\prod_{i=3}^s \binom{m_k^{i-1}}{b_i}\]
because by $2$-optimality, the number of cycles of length $2i$ containing $r$ is bounded above by the number of right-vertex sequences containing $r$ namely $m_k^{i-1}$ and we select $b_i$ such sequences for each $i$. 
By the union bound we thus conclude that
\[\PP(N^j(r))\leq \sum_{(b_3,\dots,b_s): \sum_i b_i=j} \frac{\binom{n_k -a}{cm_k^\beta-a}}{\binom{n_k}{cm_k^\beta}}\prod_{i=3}^s \binom{m_k^{i-1}}{b_i}\]
We analyze the asymptotic behavior of this quantity as $k\rightarrow \infty$ when $j\leq \eta m_k^{\beta-1}$. In particular $\sum_{i=3}^3 b_i=j$ implies that $b_i\ll m_k^{i-1}$ while $a=\sum_{i=3}^s ib_i$ implies that $a\leq s\max(b_i)\ll m_k^\beta\ll m_k^2$. These inequalities allow us to use the following standard approximations to binomial coefficients,
\[\frac{\binom{n_k -a}{cm_k^\beta-a}}{\binom{n_k}{cm_k^\beta}} \sim \left(\frac{cm_k^\beta}{m_k^2}\right)^{a} \text{ and } \binom{m_k^{i-1}}{b_i}\sim \frac{(m_k^{i-1})^{b_i}}{b_i!}\]
It turns out the resulting asymptotic expression can be computed explicitly and has a particularly simple form. By letting $p_i:= c^i m_k^{i(\beta-1)-1}$ for  $i=3,\dots, s$ and $\overline{p_i}:=p_i/\overline{p}$ where $\overline{p}=\sum_{i=3}^sp_i$ we can write
\[
\sum_{(b_3,\dots,b_s): \sum_i b_i=j} \left(\frac{cm_k^\beta}{m_k^2}\right)^{a}\prod_{i=3}^s \frac{(m_k^{i-1})^{b_i}}{b_i!} = \frac{1}{j!}\sum_{(b_3,\dots,b_s): \sum_i b_i=j} \binom{j}{b_3, b_4, \dots, b_s} \prod p_i^{b_i}=\]
\[=\frac{\overline{p}^j}{j!} \sum_{(b_3,\dots,b_s): \sum_i b_i=j} \binom{j}{b_3, b_4, \dots, b_s} \prod \overline{p_i}^{b_i} = \frac{\overline{p}^j}{j!}
\]
where the last equality occurs because the sum of the $b_i$ runs over all nonnegative indices summing to $j$ and the multinomial is a probability distribution so its probability mass function sums to one.    
Finally, we will use the fact that $j=\eta m_k^{\beta-1}$ to prove that the probability goes to zero as $k\rightarrow \infty$. By Stirling's inequality $1/j!\leq (\frac{e}{j})^j$
\[\frac{\overline{p}^j}{j!}\leq \left(\frac{e\overline{p}}{\eta m_k^{\beta-1}}\right)^j\leq \left(c^s\frac{e}{\eta} m_k^{(s-1)(\beta-1)-1}\right)^j\]
which decreases exponentially as $m_k\rightarrow \infty$ because the exponent $(s-1)(\beta-1)-1<0$ by our assumption that $\beta<1+\frac{1}{s-1}$. Since the decrease is exponential in $m_k$ the same is true if one sums over all the $m_k^2$ right vertices $r$ completing the proof of the Theorem.
\end{proof}

\begin{remark}\label{remark_stronger_exponent} The previous proof also shows that whenever the stricter condition $1<\beta<1+\frac{1}{s}$ holds the probability that likelihood that any right vertex is altered more than once vanishes (case $j=2$). Consequently, under this condition, the final degree distribution should be almost identical to the hyper-geometric in the asymptotic case. 
\end{remark}

\begin{theorem} \label{thm: rmin_main} Let $(G_k)_{k \in \NN}$, $G_k = (L_k \sqcup R_k,E_k)$ be a family of $2$-optimal, $d$-left regular and $\delta$-right regular graphs such that $|L_k| = n_k \sim m_k^2 = |R_k|^2$. For each $k$, let $X_k \subseteq L_k$ be a uniformly chosen random subset of size $\left\lfloor c\cdot m_k^{\beta}\right \rfloor<n_k$, let $H_k$ be the induced bipartite subgraph with vertices $X_k\sqcup N(X_k)$ and $T_k$ the subgraph resulting from successively removing all the left-vertices of some cycle of size $\leq 2s$. The minimum right-degree $d_{T_k}^*$ of $T_k$ satisfies 
    \[\lim_{k\rightarrow \infty}\Pr\left(d_{T_k}^* < \frac{1}{2} d \cdot cm_k^{\beta-1}\right)= 0\]
\end{theorem}
\begin{proof} Follows from combining the conclusions of Lemma~\ref{lem: rmin_selection} and Lemma~\ref{lem: rmin_alterations}.
\end{proof}

\section{Transfer bounds for expansion constants}

Suppose that either from theory or from exhaustive exploration we obtain a lower bound on one expansion constant $\alpha_G(s)$ and ask whether this leads to improved lower bounds on $\alpha_G(h)$ for larger sets, of size $h\geq s$. Our main result is that this is indeed the case. We call the resulting inequalities \emph{transfer bounds} since expansion is transferred from smaller to larger sets.

We begin with some preliminary notions. If $X\subseteq L$ is any set of size $h$, we let $H_X$ denote the bipartite subgraph of $G$ with vertices $X\sqcup N(X)$ and all edges of $G$ starting at points of $X$. Each right vertex of $H_X$ has a degree $\leq h$ and the sequence of right-degrees is summarized by the sequence of integers $(\beta_1,\dots, \beta_h)$ with
\[\beta_j:= \left|\left\{r\in N(X): \text{ $r$ has degree $j$ in $H_X$}\right\}\right|.\]
It is immediate that $|N(X)|=\sum_{j=1}^h\beta_j$ and, by double-counting the edges of $H_X$, that the equality
$\sum_{j=1}^h j\beta_j = dh$ holds.

The results in this Section depend on the crucial insight, introduced in~\cite{CCLO} that knowing the expansion constants of smaller sets leads to additional linear inequalities that the $\beta_j$ must satisfy for any set $X$ of size $h$ allowing us to use tools from linear optimization to derive bounds for $\alpha_G(h)$. More precisely, 

\begin{lemma}\label{lem: linbound} Let $s,h$ be positive integers with $s\leq h$ and for $j=1,\dots, h$ define $p_j:=1-\frac{\binom{h-j}{s}}{\binom{h}{s}}$. If $\gamma^*(h)$ is the optimal value of the linear optimization problem 
\[\min_{x\in\RR^h}
\left\{
\sum_{i=1}^h x_i :
x_j\ge0,\;
\sum_{j=1}^h x_j p_j \ge s\alpha(s),
\;
\sum_{j=1}^h jx_j = dh
\right\}.\] 
then the inequality $\alpha_G(h)\geq \gamma^*(h)/h$ holds. 
\end{lemma}
\begin{proof} Suppose $X\subseteq L$ is any subset of size $h$ and let $(\beta_1,\dots, \beta_h)$ be the right-degree counts of $H_X$ as in the previous paragraph. We want to estimate the average number of right neighbors of a subset $S\subseteq X$ of size $s$ if the set is chosen uniformly at random. So,
\[\EE[|N(S)|]=\EE\left[\sum_{r\in N(X)} 1_{r\in N(S)}\right]=\sum_{r\in N(X)} \PP\{r\in N(S)\}\]
Since $S$ is chosen uniformly at random this probability is completely determined by the right-degree of $r$. If $r$ has right-degree $j$ then the probability is $p_j:=1-\frac{\binom{h-j}{s}}{\binom{h}{s}}$ because $r\not\in N(S)$ precisely when $S$ is entirely contained in the set of $h-j$ elements of $X$ that are not neighbors of $r$.
We conclude that
\[\EE[|N(S)|]=\sum_{j=1}^hp_j\beta_j\]
and, because the average is always bounded below by the minimum, that the inequality
\[\alpha_G(s)s\leq \EE[|N(S)|] = \sum_{j=1}^hp_j\beta_j\]
holds proving that $(\beta_1,\dots, \beta_h)$ is a feasible point of the above optimization problem.
It follows that $|N(X)|=\sum_{j=1}^h \beta_j\geq \gamma^*(h)$ proving that $\alpha_G(h)\geq \gamma^*(h)/h$  since $X$ was an arbitrary set of size $h$.\end{proof}

We are now ready to prove the main result of this section. In the proof we will make repeated use of the binomial identity
\[\binom{n}{k} = \binom{n-1}{k-1}+ \binom{n-1}{k}.\]

\begin{proof}[Proof of Theorem~\ref{Thm: transfer}] The dual of the linear programming problem in Lemma~\ref{lem: linbound} is equivalent to the following two-dimensional problem
\[\max_{(\eta_1,\eta_2)\in \RR^2, \eta_1\geq 0}\left\{ \eta_1(s\alpha(s))+\eta_2(dh) : \eta_1p_j+\eta_2j\leq 1\text{ for j=1,\dots, h}\right\}\] 
First we will show that the point $(\eta_1^*,\eta_2^*)$ where the first two inequalities are active, that is the solution of the linear system
\[ 
\begin{cases}
\eta_1^*p_1+\eta_2^*1 = 1 \\
\eta_1^*p_2+\eta_2^*2 = 1 \\
\end{cases}\]
is a feasible point of the dual. Solving the linear system we conclude that $\eta_1^*:=\frac{1}{2p_1-p_2}$ and that $\eta_2^*:=\frac{p_1-p_2}{2p_1-p_2}$. The binomial identity above implies that $\eta_1^* = \binom{h}{s}/\binom{h-2}{s-2}$ and in particular that it is a positive number. Furthermore we claim that for every $j\geq 2$ the inequality 
\[\eta_1^*p_{j+1}+\eta_2^*(j+1) < \eta_1^*p_{j}+\eta_2^*j\]
holds, finishing the proof of feasibility.
This inequality is equivalent to the inequality 
\[p_2-p_1 = -\frac{\eta_2^*}{\eta_1^*} > p_{j+1}-p_j.\]
Since $p_{j+1}-p_{j} = \binom{h-j-1}{s-1}/\binom{h}{s}$ this inequality is in turn equivalent to 
\[\binom{h-2}{s-1} > \binom{h-(j+1)}{s-1}\]
which clearly holds because $j\geq 2$. Since the value of the dual objective function at any feasible point of the dual is a lower bound on the optimal function of the primal we conclude that 
\[\alpha_G(h)h\geq \eta_1^*(s\alpha(s))+\eta_2^*(hd).\]
We complete the proof by evaluating the right-hand side. Simplifying the binomial ratios above we see that $\eta_1^*=\frac{h(h-1)}{s(s-1)}$ and $\eta_2^*=\frac{s-h}{s-1}$. Dividing both sides of the previous inequality by $h$ we conclude that
\[\alpha_G(h)\geq d\frac{s-h}{s-1}+\frac{h-1}{s-1}\alpha_G(s)=d-\frac{h-1}{s-1}\left(d-\alpha_G(s)\right)\]
as claimed.
For the final inequality, if $G$ is $s$-optimal then we apply the transfer inequality we just showed with $\alpha_G(s)=d-1+\frac{1}{s}$ obtaining 
\[\alpha_G(h)\geq d-\frac{h-1}{s-1}\left(-1+\frac{1}{s}\right) = d-\frac{h-1}{s}\]
as claimed.
\end{proof}

\begin{remark} The complementary slackness conditions $\lambda_i^*x_i^*=0$, $\eta_1^*(p^t x^*-s\alpha(s))=0$ and $\eta_2^*\left(\sum jx_j^*-hd\right)=0$ imply that setting $x_j^*=0$ for $j\geq 3$ and defining $(x_1^*,x_2^*)$ as the solution of the system of equations 
\[ 
\begin{cases}
x_1^*+2x_2^* = hd \\
x_1^*p_1+x_2^*p_2 = s\alpha_G(s) \\
\end{cases}\]
leads to an optimal solution of the primal problem with the same objective function as the point $(\eta_1^*,\eta_2^*)$ showing that it is not only feasible but in fact optimal for the dual.
\end{remark}

\section{Post-quantum key exchange and optimal expanders}
\label{Sec: PQ}

In this Section we develop an extended application of small-set expanders. We begin by recalling some preliminaries from coding theory and cryptography.

As in the introduction, let $B(G)$ denote the binary code defined by a bipartite graph $G=L\sqcup R$ which is left-regular of degree $d$. The following Lemma summarizes well-known results of Sipser and Spielman for the code $B(G)$, rephrased in terms of the expansion constants of $G$. 

\begin{lemma}\label{lem: sisp} The following statements hold for any integer $h$ with $0<h<n$:
\begin{enumerate}
\item If $\alpha_G(h)>\frac{d}{2}$ then $h$ is a lower bound on the minimum distance of $B(G)$.
\item If $\alpha_G(h)>\frac{3d}{4}$ then the Sipser-Spielman bit flipping algorithm can correct any $\frac{h}{2}$ errors in linear time.
\end{enumerate}
\end{lemma}
\begin{proof} Part $(1)$ is~\cite{SiSp}[Theorem 7].$(2)$ follows from the proof of~\cite{SiSp}[Theorem 10]. 
\end{proof}
\begin{remark}
The bit-flipping algorithm is so simple that we cannot resist giving a brief description. Given $w=(x_1,\dots, x_n)\in \FF_2^{n}$ we carry out the following procedure:
\begin{enumerate}
\item Find a variable $x_i\in L$ that has strictly more unsatisfied than satisfied constraints and flip its value (replacing $x_i$ by $1+x_i$ in $\FF_2$).
\item Repeat until no such variable remains. 
\end{enumerate}
If $w=c+e$ where $c\in B(G)$ and $e$ has weight $\leq h/2$ then the algorithm is guaranteed to recove $c$ whenever $\alpha_G(h)>\frac{3d}{4}$. Furthermore, this algorithm runs in linear time (a consequence of \cite{SiSp}[Theorem 10]).
\end{remark}

Next we will use transfer inequalities and the above guarantees to argue that the codes derived from $s$-optimal expanders are potentially useful in code-based cryptography. We begin with a description of the basic problem and then prove Theorem~\ref{thm: PQK} which contains the main application.

We will now Alice and Bob wish to exchange data securely over an open channel. They can easily do so using established symmetric encryption protocols provided they both know a very large integer $K$, which we refer to as the \emph{shared key}, to be used by both parties during symmetric encoding and decoding.  Note however that any third party with knowledge of $K$ could easily read the communications between Alice and Bob.
A key-sharing protocol is a mechanism through which Alice and Bob generate a shared $K$ \emph{securely} while \emph{communicating only through the open channel}. More specifically, the following protocol is a
Niederreiter-style code-based KEM, following the general
paradigm used in modern code-based cryptography such as
BIKE \cite{BIKE20}:

\begin{enumerate}
\item Alice constructs a left-regular bipartite graph with parameters $(n,m,d)$ via its parity-check $m\times n$ matrix $H$. She also constructs and a random invertible $m\times m$ matrix $S$. Alice publishes the product $R:=SH$ and keeps the pair $(S,H)$ secret.
\item Bob generates (uniformly at random) an error vector $e$ with support size $t$. Bob publishes $c:=Re$ and the support size $t$, and keeps the vector $e$ secret.
\item Alice reads $c$, computes $S^{-1}c = He$ and uses the error-correction algorithm for the code $H$ to recover the vector $e$.

\item If step $(3)$ is successful then both Alice and Bob know $c$ and $e$ and can independently compute the key $K:={\rm Hash}(c;e)$ where ${\rm Hash}$ is any previously established hash function.
\end{enumerate}
 
\begin{remark} An eavesdropper who has access to all communications between Alice and Bob would have to be able to compute $e$ from $c$ and $R$ in order to obtain the key. The eavesdropper must try to do so knowing only the "scrambled" version $R$ of the code $H$. It is known that finding the nearest vector in a subspace $V\subseteq \FF_q^n$ to a given vector $w$ is a difficult computational problem if $V$ is sufficiently generic \cite{Ber}[Sect-III] . The hope is that the scrambled code resembles, in the eyes of an eavesdropper, a generic code sufficiently so that decoding becomes difficult.

The key exchange above reached NIST PQ level three~\cite{BIKE20} as post-quantum cryptography alternative candidate, meaning that it is currently believed to be safe, for a suitable choice of code, even if the eavesdropper is in possession of a quantum computer. \end{remark} 
  
There are a few standard attacks that an eavesdropper could use. Their aim is to recover $e$ from $c$ and $R$ and thus capture the key $K$. \begin{enumerate}
\item \emph{Brute force attack.}  The attacker selects a possible support set $X$ uniformly at random, defines $y:=\sum_{j\in X} e_j$ and verifies whether $Ry=c$. The probability of success is $1/\binom{n}{h}$ so the expected number of attempts before success is $\binom{n}{h}\sim 2^{h\log_2\left(n/h\right)}$.

\item \emph{Information set decoding.} From the public knowledge of $R$ the attacker can compute the dimension $k$ of the code. The attacker selects a random support set $I\subseteq [n]$, $|I|=k$ and attempts to solve the restricted linear system $R_Iy=s$ where $R_I$ is the restriction of the matrix $R$ to the columns indexed by $I$. If $I$ is an information set (i.e. if $R_I$ is invertible) then ${\rm supp}(e)\subseteq I$ implies $y=e$. It follows that the success probability of a given information set $I$ is given by $\binom{k}{h}/\binom{n}{h}\sim 2^{-h\log_2(n/k)}$ so the expected number of attempts before success of this strategy is of the order of $2^{h\log_2(n/h)}$.

\item \emph{Low weight dual constraints attack.} The code is obscured via left multiplication by the random invertible matrix $S$. This obfuscation method however, leaves the dual code $C^{\perp}$ exposed, where 
\[C^{\perp}:=\left\{v\in \FF_2^n: \forall x\in C\left(v^tx=0\right) \right\}\]
because $C^{\perp}={\rm rowspan}(H)={\rm rowspan}(R)$ (see \cite{MWS}[Ch.5]). If $\lambda^tR=v\in C^{\perp}$ is a dual word of low weight then $e$ must satisfy the linear equation $ve=\lambda^tc$ and it is possible to use such equations to decrease the complexity of brute force attacks by searching exhaustively over the part of $e$ with the same support as the short dual words and then piecing together $e$ gradually from these pieces. Although a useful general theoretical description of the dual code is still lacking, certain general rules are practically advisable to mitigate these attacks: The graph should have comparatively large minimum right-degree and few short cycles, because violating either of these conditions would lead to dual words of small weight. In any case scrambling by $S$ makes finding such short words more difficult.
\end{enumerate}

The main result of this section is Theorem~\ref{thm: PQK} below which proves that $s$-optimal codes can be used to simultaneously ensure protection against these standard attacks and Alice's success in the critical step $(3)$, in the key-exchange procedure above.

\begin{lemma}\label{lem: existence} Let $s\geq 2$ be an integer. For each prime power $q$ there exists $j:=j(q,s)$ and an $s$-optimal left regular bipartite graph with parameters $m=q^j$, $d=q+1$ and $n \geq 2m$.
\end{lemma}
\begin{proof} Apply the random selection strategy in the proof of Theorem~\ref{Thm: s_optimal_existence} with parameter $c=2$ to the line-point correspondences in projective spaces of dimension $j$ from Example~\ref{example: PS}.
\end{proof}

\begin{theorem}\label{thm: PQK} If the $s$-optimal graphs of Lemma~\ref{lem: existence} are used in the key exchange protocol above then the following statements hold:
\begin{enumerate}
\item Alice can recover $e$ in linear time if $|{\rm supp}(e)|\leq \frac{s(q+1)}{8}$ and
\item If $|{\rm supp}(e)|\geq \frac{s(q+1)}{8}$ then the work required by an eavesdropper to do either brute force or information set decoding to recover $e$ from $c$ grows exponentially in the product $sq$.
\end{enumerate}
\end{theorem}
\begin{proof} $(1)$ By Lemma~\ref{lem: sisp}, we know that bit flipping is capable of correcting $h/2$ errors in linear time when  $\alpha_G(h)> 3d/4$. Since the graphs under consideration are $s$-optimal the transfer bounds of Theorem~\ref{Thm: transfer} imply that for any $h \geq s$ we have $\alpha_G(h)\geq d-\frac{h-1}{s}$. We conclude that Alice can recover $h/2=sd/16=s(q+1)/16$ errors uniquely and in particular can recover $e$ in linear time whenever its support is of size at most $s(q+1)/16$. $(2)$ If the support size $h\geq \frac{s(q+1)}{8}$ then the proof of work inequalities in the previous paragraph imply that a brute force attack requires work in the order of $2^{\frac{s(q+1)}{8}}$. 
Since for large $q$ the inequality
\[\frac{n}{h} \geq \frac{8q^j}{s(q+1)}\geq q^{j-2}\]
holds, we also conclude that an information set decoding attack requires work in the order of $2^{\frac{s(q+1)(j-2)}{8}}$ completing the proof. 
\end{proof}

\begin{examplex}
    As a concrete instance of the previous construction, 
     fixing $q = 2$ and $k = 10$, we are able to generate $3$-optimal graphs with minimum right degree $5$ (see figure \ref{fig:histogram}) consistently by picking $\beta = 1.25 < 1.5$ and a left-to-right constant $c = 1.25$. This way we are able to gain dual code robustness with the minimum right degree growing like $\sim \sqrt[4]{m_k}$. Furthermore, this histogram is consistent with remark $\ref{remark_stronger_exponent}$.
\end{examplex}

\begin{figure}
    \centering
    \includegraphics[width=0.5\linewidth]{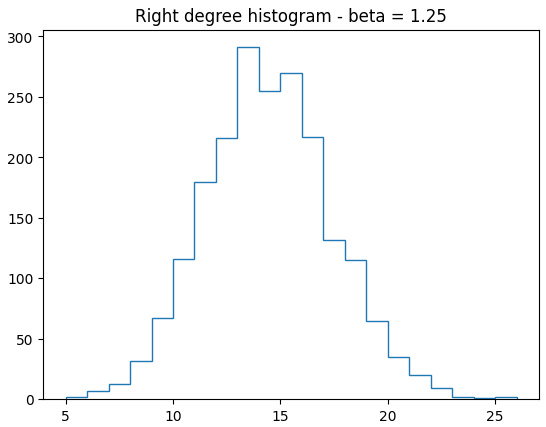}
    \caption{Histogram of degree distribution. $\beta = 1.25, k = 10, q = 2, c = 1.25$. }
    \label{fig:histogram}
\end{figure}

\begin{remark}
    Because of Lemma \ref{lem: exps}, we know the size of the graph will be of the order $(q+1)^{s} = 2^{s\log(q+1)}$ which is quite large as $s,q$ grow. However, note that there is still a significant gap between this and the brute force or information decoding attack $\sim 2^{\Theta(sq)}$. Increasing $q$ thus provides more security much faster than increasing graph size, although naturally this increases the degree which in turn slows down decoding.
\end{remark}

{\bf Acknowledgments.} We thank Val\'erie Gauthier for stimulating our interest in post-quantum cryptography. We thank Mateo Diaz and César Galindo for useful conversations during the completion of this project. M. Fiori, P. Raigorodsky and M. Velasco are partially supported by the Fondo Clemente Estable (ANII, Uruguay) grant FCE-1-2023-1-176172. Tristram Bogart was supported by internal research grant INV-2025-213-3438 from the Faculty of Sciences of the Universidad de los Andes.

\FloatBarrier

\end{document}